\documentclass[12pt]{amsart}
\usepackage{amscd,amssymb}
\usepackage{stmaryrd}
\usepackage[mathscr]{eucal}
\usepackage[english]{babel}

\input xy
\usepackage[all]{xy}
\usepackage{graphicx}
\usepackage{amsmath}
\usepackage{comment}
\usepackage{enumerate}

\usepackage{subfigure}
\usepackage{psfrag}
\usepackage[bf]{caption}
\usepackage{epsfig}
\usepackage{epstopdf}
\usepackage{hyperref}

\usepackage[latin1]{inputenc}%
\usepackage{xcolor}

\newlength{\itemlaenge}

\newtheoremstyle{mytheorem}
  {}
  {}
  {\slshape}
  {}
 {\scshape}
  {.}
  { }
  {}

\newtheoremstyle{mydefinition}
  {}
  {}
  {\upshape}
  {}
  {\scshape}
  {.}
  { }
  {}

\theoremstyle{mytheorem}
\newtheorem{lemma}{Lemma}[section]
\newtheorem{prop}[lemma]{Proposition}
\newtheorem*{prop*}{Proposition}
\newtheorem{prop_intro}{Proposition}

\newtheorem{cor_intro}[prop_intro]{Corollary}

\newtheorem{thm_intro}[prop_intro]{Theorem}

\newtheorem*{thm*}{Theorem}
\theoremstyle{mydefinition}
\newtheorem{rem}[lemma]{Remark}
\newtheorem*{rem*}{Remark}
\newtheorem{rem_intro}[prop_intro]{Remark}

\newtheorem*{notation*}{Notation}

\newtheorem*{warning*}{Warning}

\newtheorem{defi}[lemma]{Definition}
\newtheorem*{defi*}{Definition}

\numberwithin{equation}{section}

\newcommand{\bqn}{\begin{equation*}}
\newcommand{\eqn}{\end{equation*}}
\newcommand{\bq}{\begin{equation}}
\newcommand{\eq}{\end{equation}}
\newcommand{\ba}{\begin{align}}
\newcommand{\ea}{\end{align}}
\newcommand{\be}{\begin{enumerate}}
\newcommand{\ee}{\end{enumerate}}

\newcommand{\thismonth}{\ifcase\month 
  \or January\or February\or March\or April\or May\or June%
  \or July\or August\or September\or October\or November%
  \or December\fi}

\newcommand{\EE}{\mathcal{E}}
\newcommand{\En}{\mathcal{E}^{(n)}}

\newcommand{\Ek}{\mathcal{E}^{(k)}}

\newcommand{\Evn}{\mathcal{E}_v^{(n)}}
\newcommand{\Evk}{\mathcal{E}_v^{(k)}}
\newcommand{\Evkone}{\mathcal{E}_v^{(k-1)}}
\newcommand{\Evone}{\mathcal{E}_v^{(1)}}
\newcommand{\TT}{\mathcal{T}}
\newcommand{\aut}{\text{Aut}}

\newcommand{\VV}{\mathcal{V}}
\def\hb{{\rm H}_{\rm b}}
\def\h{{\rm H}}
\def\pr{\mathrm{pr}}

\def\PSL{\mathrm{PSL}}
\def\stab{\operatorname{Stab}}
\def\FF{\mathbb F}

\def\QQ{\mathbb Q}
\def\ZZ{\mathbb Z}
\def\Gg{\mathcal G}
\def\Ll{\mathcal L}

\author[]{Alessandra Iozzi}
\address{Department Mathematik\\
         ETH Zentrum\\
         8092~Z\"urich\\
         }
\email{alessandra.iozzi@math.ethz.ch}
\author[]{Cristina Pagliantini}
\address{Department Mathematik\\
         ETH Zentrum\\
         8092~Z\"urich\\
        }

\email{cristina.pagliantini@math.ethz.ch}

\author[]{Alessandro Sisto}
\address{Department Mathematik\\
         ETH Zentrum\\
         8092~Z\"urich\\
        }
\email{sisto@math.ethz.ch}
\title[Quasimorphisms on trees]%
      {Characterising actions on trees yielding non-trivial quasimorphisms}

\thanks{All authors were supported by Swiss National Science Foundation project 144373.
}
\begin{document}

\begin{abstract}
We study the construction of quasimorphisms on groups acting on trees introduced by
Monod and Shalom, that we call \emph{median} quasimorphisms, and in particular we fully characterise actions on trees that give rise to non-trivial median quasimorphisms.
Roughly speaking, either the action is highly transitive on geodesics, it fixes a point in the boundary, or there exists an infinite family of non-trivial median quasimorphisms.
In particular, in the last case the second bounded cohomology of the group is infinite dimensional as a vector space.
As an application, we show that a cocompact lattice in a product of trees only has trivial quasimorphisms if and only if both closures of the projections on the two factors are locally $\infty$-transitive. 
\end{abstract}

\maketitle

\section{Introduction}

\medskip
Since its inception in the 80s, bounded cohomology has proven itself to be a very efficient tool
to approach rigidity questions (see for example
\cite{Gromov, Ghys, Matsumoto, Burger_Iozzi_supq, Burger_Iozzi_Wienhard_toledo, Bader_Furman_Sauer}).
While in degree 0 and 1 there are no enlightening differences between bounded and ordinary group cohomology,
in degree two the study of the natural comparison map $\hb^2\to\h^2$ has also proven
to be very fruitful.  For example the surjectivity of the comparison map with all coefficients is 
a characterisation of Gromov hyperbolicity \cite{Mineyev_1, Mineyev_2}, 
while the injectivity is equivalent to the vanishing of the stable commutator length \cite{Bavard}.  
It is hence natural that a lot of attention has been devoted 
to the study of the kernel of the comparison map, that is to say the space of quasimorphisms.

\medskip
In this paper we study the space of quasimorphisms of a group $\Gamma$ acting by automorphisms of a simplicial tree $\TT$ and we denote $\rho:\Gamma\rightarrow \text{Aut}(\TT)$ such an action.  On the one hand we give a trichotomy to show that, roughly speaking,
either the action of $\Gamma$ on $\TT$ is highly transitive on geodesics, it fixes a point 
in the boundary $\partial\TT$ of $\TT$, or there exists an infinite family of non-trivial 
quasimorphisms on $\Gamma$.  As an application, we prove the converse of a result of
Burger and Monod, \cite[Corollary~26]{Burger_Monod_GAFA}, 
thus giving a full characterization of locally $\infty$-transitive groups in terms of quasimorphisms.

\medskip


More precisely, if $n\in \mathbb{N}$, we call $n$-geodesic a geodesic segment of length $n$.
Let $\En$ be the set of connected \emph{oriented} $n$-geodesics in $\TT$ and 
let us consider the isometric action of $\Gamma$ on $\ell^1(\En)$, given by 
$\pi(g)f(s):=f(g^{-1}s)$.
Let us fix a vertex $v\in \TT$.
We start by recalling the construction of a cocycle due to Monod and Shalom, \cite{Monod_Shalom},
which is a discrete version of the Gromov--Sela cocycle.
We define a map 
$\omega:\Gamma\to\ell^1(\En)$
by 
\begin{equation*}
\omega(g):=\chi_{\llbracket v,gv\rrbracket}-\chi_{\llbracket gv,v \rrbracket },
\end{equation*}
where $\llbracket x,y \rrbracket$  is the set of connected geodesic segments of length $n$ 
contained in the geodesic $[x,y]$ and oriented as $[x,y]$.
It is easy to see that while $\omega$ is not bounded as a function of $g$, 
\begin{equation*}
  \sup_{g,h\in\Gamma}\|\delta^1\omega(g,h)\|_1<\infty,
\end{equation*}
as $\delta^1\omega(g,h):=\pi(g)\omega(h)-\omega(gh)+\omega(g)$
is supported only on the geodesic paths that go through the median 
of the vertices $v,gv$ and $ghv$\footnote{This means that the cocycle 
$\delta^1\omega:\Gamma\times\Gamma\to\ell^1(\En)$ defines 
a bounded cohomology class $[\delta^1\omega]\in\hb^2(\Gamma,\ell^1(\En))$.}. 
Given an oriented $n$-geodesic $s$, we define the {\em median quasimorphism}  
$f_{\rho(\Gamma)s}:\Gamma\to\mathbb{Z}$ 
to be the evaluation of $\omega$ on the characteristic function $\chi_{\rho(\Gamma)s}\in\ell^\infty(\En)$
of the $\Gamma$ orbit of $s$
\begin{equation}\label{eq:fL}
    f_{\rho(\Gamma)s}(g)
:=\langle\omega(g),\chi_{\rho(\Gamma)s}\rangle
 =\sum_{\gamma\in \rho(\Gamma)s}\chi_{\llbracket v,gv\rrbracket}(\gamma)-\chi_{\llbracket gv,v \rrbracket }(\gamma).
\end{equation}

Then, in Section~\ref{mainthm} we prove the following result:

\begin{thm_intro}\label{thm:classification}
 Suppose that $\Gamma$ acts minimally on the tree $\TT$ and that every vertex of $\TT$ has valence greater than 2. Then exactly one of the following holds.
 
 \begin{enumerate}[$(i)$]
\item There exists $l\in \{1,2\}$ so that for each $n$, $\Gamma$ acts transitively on
$$\{[x,y]\in\En: \ d(x,v) \equiv 0\, (l)\},$$
where $v$ is any vertex of $\TT$.
\item $\Gamma$ fixes a point $a$ in the boundary $\partial \TT$ of $\TT$.
\item There exists an infinite family $\{[\delta^1f_{\rho(\Gamma)s_n}]\}_{n\in\mathbb N}$  of linearly independent
coclasses in $\hb^2(\Gamma,\mathbb{R})$. In particular, there exists an injective $\mathbb{R}$-linear map $\ell^1\rightarrow \hb^2(\Gamma,\mathbb{R})$.
 \end{enumerate}
\end{thm_intro}

\begin{rem_intro}\label{val2}
 The conditions that $\Gamma$ acts minimally on the tree $\TT$ and that every vertex of $\TT$ has valence greater than 2 are not restrictive because given any action on a tree $\TT$ one can consider a minimal subtree $\TT'$, which exists provided that $\Gamma$ does not consist entirely of parabolic elements fixing the same point at infinity \cite[Corollary 3.5]{Tits}. There are no vertices of valence 1 in $\TT'$ unless $\TT'$ consists of a single edge only, and vertices of valence 2 can be removed by regarding each maximal path containing only vertices of valence 2 in their interior as a single edge. The latter procedure gives a well-defined tree provided that $\TT'$ is not a subdivision of $\mathbb R$.
 \end{rem_intro}


If $\TT$ is a locally finite tree, recall that a group $H<\aut(\TT)$ is \emph{locally $\infty$-transitive} if the stabiliser of each vertex is transitive on the $n$-spheres for all $n$.  
Local $\infty$-transitivity implies already very strong properties on the tree, 
that will be either regular or bipartite.  
Moreover, in some sense locally $\infty$-transitive groups are the most rigid.  
For example if $H<\aut(\TT)$ is a closed subgroup that is locally $\infty$-transitive, then for all $N\trianglelefteq H$, $N\neq\{e\}$,
the quotient $H/N$ is compact, \cite[Proposition~3.1.2]{Burger_Mozes_gat}.  
Or else, if $\Gamma<H_1\times H_2$ is a cocompact lattice with $H_i<\aut(\TT_i)$
and $H_i$ is locally $\infty$-transitive, then $\Gamma$ satisfies the Normal Subgroup Theorem, 
that is every normal subgroup is finite or of finite index, \cite{Bader_Shalom}.  

We prove in Section \ref{sec:transitivity}:

\begin{cor_intro}\label{cor:locally_inf_trans} Let $\TT_i$, $i=1,2$ be regular trees of valence greater than $2$ and 
let $\Gamma<\aut(\TT_1)\times\aut(\TT_2)$ be a cocompact lattice. Also, let $\pr_i:\aut(\TT_1)\times \aut(\TT_2)\rightarrow \aut(\TT_i)$ be the projection maps.  Then the following are equivalent:
\begin{enumerate}
\item the $H_i:=\overline{\pr_i(\Gamma)}^Z$ are locally $\infty$-transitive, $i=1,2$;
\item any quasimorphism $\Gamma\to\mathbb{R}$ is bounded;
\item any median quasimorphism $\Gamma\to\mathbb{R}$ is bounded.
\end{enumerate}
\end{cor_intro}

As a consequence of this corollary we can construct infinite families of examples of groups for which
Theorem~\ref{thm:classification}-(iii) holds, as appropriate extensions of irreducible cocompact lattices 
in the product of $\PSL(n,\mathbb{Q}_p)\times\PSL(n,\mathbb{Q}_\ell)$, 
with $n\geq3$, $p,\ell$ prime (see Section~\ref{examples}).

\medskip
Theorem~\ref{thm:classification} applies in particular to amalgamated products and HNN extensions. Fujiwara~\cite{Fujiwara} showed that, 
under mild hypotheses, these have infinite dimensional second bounded cohomology. 
We show that Fujiwara's result can be deduced from ours, that gives in fact a strictly 
larger class of amalgamated products with infinite dimensional second bounded cohomology. 
For instance the group $S_3\ast_{\mathbb{Z}_2}\mathbb{Z}_4$, 
where $S_3$ is the symmetric group on $3$ elements, satisfies condition (iii) of Theorem~\ref{thm:classification} 
but is not covered by Fujiwara's result, \cite[Theorem~1.1]{Fujiwara}.
See Section~\ref{examples} for details.

Another very large class of groups known to have infinite dimensional second bounded cohomology 
is that of \emph{acylindrically hyperbolic groups} \cite{HullOsin} (see also \cite{FriPoSi,BBF2}), 
and under rather general conditions a group acting on a tree is acylindrically hyperbolic \cite[Theorem 2.1]{MO}. 
However, we remark that groups satisfying condition (iii) of Theorem~\ref{thm:classification} include
examples that are not covered by \cite{MO} (once again, see Section~\ref{examples} for details).

\section{Notation}\label{sec:notation}
We fix the notation of the theorem, until the end of Section \ref{remarks}.

We fix a vertex $v$ of $\TT$ as basepoint. All geodesics have vertices of $\TT$ as endpoints.

\begin{defi}
The \emph{orbit length} (or \emph{o-length}) of a geodesic $\gamma$  is the number of vertices
in $\rho(\Gamma) v$ crossed by $\gamma$ minus one. We denote the o-length by $|\cdot|_o$.  
\end{defi}

An $o(n)$-geodesic is a geodesic segment of orbit length $n$ with vertices in $\rho(\Gamma) v$ as endpoints. If we do not need to emphasize the length we will use the notation o-geodesic.
We call \emph{o-edge} an $o(1)$-geodesic and \emph{o-vertex} a vertex in the orbit $\rho(\Gamma) v$.
The set of oriented $o(n)$-geodesics is denoted by $\Evn$.

\begin{defi}
Let $\gamma_1$ and $\gamma_2$ be two o-geodesics. We say that $\gamma_1$ is \emph{chainable} 
with $\gamma_2$ if $\xi \gamma_1$ concatenates with $g \gamma_2$ in an o-vertex for some 
$g, \xi \in \Gamma$ and they
only intersect in such o-vertex.
\end{defi}

\section{Proof of Theorem~\ref{thm:classification}}\label{mainthm}

\subsection{Mutual exclusiveness}
First of all, we observe that cases (i), (ii) and (iii) are mutually exclusive. It is clear that (i) and (ii) cannot simultaneously occur.

Suppose now that (i) holds, and let us show that for any oriented geodesic $s$, $f_{\rho(\Gamma)s}$ is identically $0$. In fact, the number of translates of $s$ contained in an o-geodesic $\gamma$ only depends on the o-length of $\gamma$, and hence it coincides with the corresponding number for $\gamma^{-1}$. By definition of $f_{\rho(\Gamma)s}$, this implies that $f_{\rho(\Gamma)s}$ is identically $0$.

If (ii) holds, any oriented geodesic is obtained concatenating two (possibly trivial) subgeodesics $\gamma_1,\gamma_2$ so that $\gamma_1$ and $\gamma_2^{-1}$ are each contained in some ray towards the fixed point $a\in\partial\TT$. Furthermore, this decomposition is preserved by the action of $\Gamma$. If the decomposition for the oriented geodesic $s$ is trivial, meaning that one of the two subgeodesics is trivial, then $f_{\rho(\Gamma)s}$ is $0$ by an argument similar to the one we gave above. If the decomposition is non-trivial, any oriented geodesic contains at most one translate of $s$, whence $f_{\rho(\Gamma)s}$ is bounded.

\subsection{The strategy}
We now have to show that when (i), (ii) do not hold, then (iii) must hold.

Here is an informal description of the strategy that we will follow. We will define an equivariant labelling on the oriented o-geodesics of $\TT$, the label of an oriented o-geodesic being a word in the alphabet $\{a,b,c\}$. We will show that certain ``sufficiently complicated'' words can be realised as labels of oriented o-geodesics (when (i) and (ii) do not hold). The main property of such words is that they only have short common subwords compared to their length, which translates into o-geodesics realizing such words having short overlaps compared to their o-length. This property will allow us to control the counting procedure that defines $f_{\rho(\Gamma)s}$ and in particular to show, for suitable choices of $s$, that the homogenization of $f_{\rho(\Gamma)s}$ is not a homomorphism.

\subsection{Labellings}\label{sub:labellings}
A \emph{labelling} for us will consist of the following data. First, a positive integer $k$. Second, the choice, for each orbit of o(k)-geodesics, of a letter from $\{a,b,c\}$, its label. With a slight abuse, we sometimes refer to the label of an o(k)-geodesic to mean the label of its orbit. Finally, the final piece of data is the assignment of a word to each o-geodesic according to the following procedure. Let $\gamma$ be an o-geodesic. Let $\gamma_1$ be the 
initial oriented subpath of $\gamma$ of orbit length $k$, we define the first letter of the label of $\gamma$ to be the label of $\gamma_1$.
More, in general, let $\gamma_i$
be the oriented subpath of $\gamma$ of orbit length $k$ that
starts from the $i$-th o-vertex of $\gamma$. Then the $i$-th letter of the label of $\gamma$ is the label of $\gamma_i$. 
Once again, we call the word associated to an o-geodesic the label of the o-geodesic.
Notice that if an o-geodesic has orbit length $L\geq k$ then the associated word has length $L-k+1$.
(The label of o-geodesics of orbit length less than $k$ is empty.)

\subsection{Choice of words}
We wish to realise certain words as labels of o-geodesics. The words satisfy the following requirements.

If $w$ is a word in the alphabet $\{a,b\}$ we denote by $\overline{w}$ the word obtained reading $w$ from right to left.

\begin{lemma}\label{lem:triple}
There exists a family of words $\{w_{ni}\}_{n\geq 1, i=1,2,3}$ in the alphabet $\{a,b\}$ such that
\begin{enumerate}[(i)]
\item each $w_{ni}$ is the concatenation of words of the form $ab^N$ for $N\geq 2$;
\item for each integer $k$ there exists $n_0(k)$ so that if $n,m\geq n_0(k)$ and if $w_{ni}$ and $w_{mj}$ share a subword of length at least $\min\{|w_{ni}|+k-1,|w_{mj}|+k-1\}/10-k$ then $n=m$ and $i=j$; 
\item Once again for $n,m\geq n_0(k)$, $w_{ni}$ does not share a common subword of length at least $\min\{|w_{ni}|+k-1,|\overline{w_{mj}}|+k-1\}/10-k$ with $\overline{w_{mj}}$;
\end{enumerate}
\end{lemma}
\begin{proof}
Let us pick $v=100(3n+i)$ with $n\geq 1$ and $i\in \{1,2,3\}$. 
We choose $w_{ni}=ab^vab^{v+1}\cdots ab^{v+100}$.
Since all the exponents of the b's are different the family of words $\{w_{ni}\}$ easily satisfies the lemma. 
\end{proof}
We say that a word in $\{a,b\}$ is \emph{good} if it satisfies the first item. In the next subsection we will exhibit, in all cases where (i) and (ii) of Theorem~\ref{thm:classification} do not apply, a labelling so that any good word can be realised as the label of some o-geodesic. The reader might wish to skip this on first reading.

\begin{rem}\label{rem:inverses}
 We will also ensure that the labellings satisfy one of the following two conditions.
 \begin{enumerate}
  \item Either we assign the label $a$ (resp.~$b$) to a unique orbit of o-geodesics, or
  \item the inverse of any element labelled $a$ is not in an orbit labelled $b$.
 \end{enumerate}
\end{rem}

\subsection{Definition of the labelling}
The labellings have to be defined in different ways depending on the properties of the action of $\Gamma$. We also show case by case why any good word can be realised as an oriented o-geodesic.

\par\medskip

\emph{Case (1). Suppose that $\Gamma$ does not act transitively on the set $\Evone$ of o-edges, and furthermore that there exist representatives $e_1,e_2, e_3$ of three distinct orbits of o-edges that only meet at their common starting point $v$.}

\par\medskip

The first required piece of data of the labelling is a positive integer: we pick $k=1$.

Let us point out the following fact that will be used in this subsection.
\begin{rem}\label{rem:chainable}
Every incoming o-edge $e$ in $v$ is chainable with at least two o-edges among $e_1$, $e_2$ and $e_3$. Indeed, the o-edge $e$ (the thickened one in Figure~\ref{chainable}) could share its final edge with at most one among  $e_1$, $e_2$ and $e_3$. 
\begin{figure}[!htbb]
\centering
\psfrag{a}{$e_1$}
\psfrag{c}{$e_2$}
\psfrag{b}{$e_3$}
\psfrag{e}{$e$}
\includegraphics[width=0.15\textwidth]{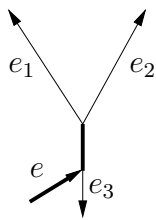}
\hspace{5mm}
\caption{$e$ is chainable with $e_1$ and $e_2$}
\label{chainable}
\end{figure}

Similarly, if we suppose that two o-edges only meet at their common starting point $v$, then every incoming o-edge $e$ in $v$ is chainable with at least one of those (see Figure~\ref{chainable2}).
\begin{figure}[!htbb]
\centering
\psfrag{b}{$e_1$}
\psfrag{c}{$e_2$}
\psfrag{e}{$e$}
\includegraphics[width=0.15\textwidth]{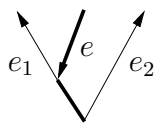}
\hspace{5mm}
\caption{$e$ is chainable with $e_2$}
\label{chainable2}
\end{figure}
In particular, if both the emanating o-edges are in the orbit of $e_1$, for instance, then every incoming o-edge in $v$ is chainable with $e_1$.
\end{rem}   

Let us now form a directed graph $\Lambda$ (possibly with loops) with vertices marked $1,2,3$ and edges as follows. If, for some $i,j$, $e_i$ is chainable with $e_j$, then connect $i$ and $j$ with a directed edge $i\rightarrow j$ in the graph $\Lambda$. By Remark~\ref{rem:chainable} it follows that at least two edges emanate from any vertex of $\Lambda$.

\begin{lemma}
 Up to permuting the indices, either $\Lambda$
 \begin{enumerate}[(a)]
  \item contains the subgraph described in Figure \ref{fig:cycle}-(a), or
  \item it is the graph of Figure \ref{fig:cycle}-(b), or
  \item it is the graph of Figure \ref{fig:cycle}-(c).
 \end{enumerate}
 \begin{figure}[!htbb]
\centering
\psfrag{1}{$1$}
\psfrag{2}{$2$}
\psfrag{3}{$3$}
\subfigure[]%
{\includegraphics[width=0.25\textwidth]{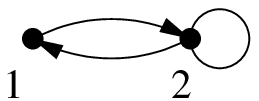}}
\hspace{5mm}
\subfigure[]%
{\includegraphics[width=0.2\textwidth]{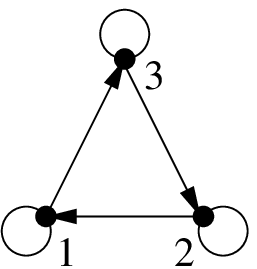}}
\hspace{5mm}
\subfigure[]
{\includegraphics[width=0.2\textwidth]{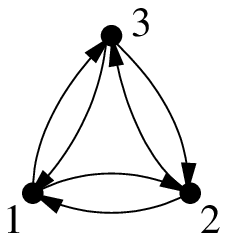}}
\hspace{5mm}
\caption{Directed graph $\Lambda$}
\label{fig:cycle}
\end{figure}
\end{lemma}
\begin{proof}
Let us suppose that $\Lambda$ does not contain a subgraph of type bigon-loop  (Figure \ref{fig:cycle}-(a)).
Let us focus on the vertex $j\in\{1,2,3\}$. Either one edge emanating from $j$ is a loop
or no edge is a loop.
If  a loop is based at $j$ then the same happens at all vertices of $\Lambda$ otherwise 
we can realise a bigon-loop. Hence,
we have the configuration described in Figure \ref{fig:cycle}-(b).
On the other hand, if no edge emanating from $j$ is a loop then there are no loops in $\Lambda$, and the only possible configuration is the one described in Figure \ref{fig:cycle}-(c).
\end{proof}

\emph{Case (1a).} Suppose that the graph $\Lambda$ contains a bigon-loop as in Figure
\ref{fig:cycle}-(a).
\par\medskip
Assign label $a$ to the orbit of the o-edge $e_1$ and label $b$ to the orbit of the o-edge $e_2$. Finally, assign label $c$ to all other orbits.

Figure \ref{fig:cycle}-(a) implies that $e_2$ is chainable with $e_1$ and this realises the syllable $ba$. Similarly, $e_1$ is chainable with $e_2$ so that we have the syllable $ab$. Finally, $e_2$ is chainable with $e_2$, which corresponds to $bb$. 
Therefore all good words can be realised as labels of o-geodesics.

\par\medskip
\emph{Case (1b).} The graph $\Lambda$ is of type described in Figure \ref{fig:cycle}-(b).
\par\medskip
Assign label $b$ to the orbit either of the o-edge $e_2$ or of the o-edge $e_3$. Also, assign label $a$ to the orbit of the o-edge $e_1$ and label $c$ to all other orbits. 

Similarly to Case (1a), Figure \ref{fig:cycle}-(b) implies that  all good words can be realised as labels of o-geodesics.

Recall that we need to verify that our labelling satisfies one of the two conditions listed in Remark~\ref{rem:inverses}. Since condition (1) does not hold, let us verify that the inverse of any element labelled $a$ is not in an orbit labelled $b$ (condition (2)).
First of all notice that the only possible configuration of the o-edges associated to Figure~\ref{fig:cycle}-(b)
is described by Figure~\ref{case1b}. 
\begin{figure}[!htbb]
\centering
\psfrag{a}{$e_1$}
\psfrag{c}{$e_3$}
\psfrag{b}{$e_2$}
\includegraphics[width=0.4\textwidth]{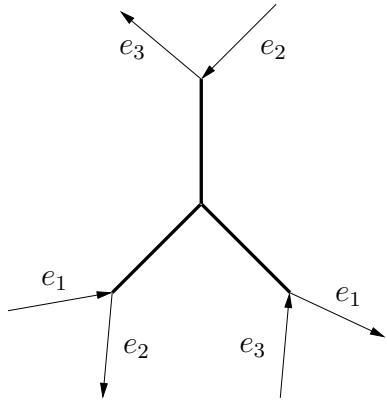}
\hspace{5mm}
\caption{configuration of o-edges associated to Figure~\ref{fig:cycle}-(b)}
\label{case1b}
\end{figure}
Therefore it can be easily seen that 
the inverse of $e_1$ is neither in the orbit of $e_2$ nor in the orbit of $e_3$, for otherwise in both cases  $e_2$ would be chainable with $e_3$. 

\par\medskip
\emph{Case (1c).} The graph $\Lambda$ is of type described in Figure \ref{fig:cycle}-(c).
\par\medskip
Assign label $a$ to  the orbit of the o-edge $e_1$. Also, assign label $b$ to the orbits of the o-edges $e_2$ and  $e_3$, and label $c$ to all other  orbits. 
Then similarly to Case (1a) we can realise the syllables $ab, ba, bb$ concatenating elements in the orbits of $e_1$, $e_2$, $e_3$. 

Moreover, as in Case (1b), let us verify that condition (2) of Remark~\ref{rem:inverses} holds.
The inverse of $e_1$ is neither in the orbit of $e_2$ nor in the orbit of $e_3$,
for otherwise $e_2$ (resp.~$e_3$) would be chainable with $e_2$ (resp.~$e_3$). Therefore the orbit of $e_1^{-1}$ is labelled with an $a$ or with a $c$. 


This completes the definition of the labelling in Case (1). 


\par\medskip

\emph{Case (2). Suppose that $\Gamma$ does not act transitively on the set $\Evone$ of o-edges, and that there do not exist o-edges $e_1,e_2,e_3$ as in Case (1).}
\par\medskip

As in Case (1) we choose as positive integer $k=1$.

 Since the valence of $v$ is at least $3$, there exist o-edges $e_1$,
$e'_1=g(e_1)$ with $g\in \Gamma$ and $e_2\notin \rho(\Gamma)e_1$ emanating from $v$ and pairwise only intersecting in $v$, see Figure \ref{fig:triple}.
\begin{figure}[!htbb]
\centering
\psfrag{a}{$e_1$}
\psfrag{c}{$e_1'$}
\psfrag{b}{$e_2$}
\includegraphics[width=0.15\textwidth]{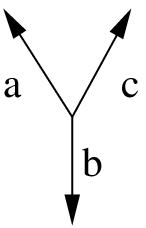}
\hspace{5mm}
\caption{Case (2)}
\label{fig:triple}
\end{figure}

\emph{Case (2a).} Suppose that $e_1$ is chainable with $e_2$.
\par\medskip
As emphasized in Remark~\ref{rem:chainable}, any o-edge is chainable with $e_1$. In this case assign label $a$ to the orbit of the o-edge $e_2$ and label $b$ to the orbit of the o-edge $e_1$. Moreover, assign  label $c$ to all other orbits. 
Then the above chainable pairs allow us to realise the syllables $ab$, $ba$, $bb$.

\par\medskip

\emph{Case (2b).} Suppose that $e_1$ is not chainable with $e_2$, but that $e_1$ is chainable with $e_1^{-1}$.
\par\medskip
This case cannot occur. In fact, by hypothesis, there are two o-edges in the orbit of $e_1$ only intersecting in their common final point $v$. This easily implies that $e_1$ is chainable with any o-edge.  
 
\par\medskip
\emph{Case (2c).} Suppose that $e_1$ is not chainable with $e_2$, and $e_1$ is not chainable with $e_1^{-1}$. We show that condition (ii) of Theorem~\ref{thm:classification} applies.

We prove that one can put an orientation on $\TT$ with the following properties:
\begin{itemize}
 \item the (oriented) edge $e$ is oriented positively (resp. negatively) if and only if it is contained in an o-edge in the orbit of $e_1$ (resp. $e_1^{-1}$).
 \item Every o-vertex has exactly one incoming edge.
\end{itemize}

Let us start with the remark that $e_1$ is not in the orbit of $e_1^{-1}$, for otherwise $e_1$ would be chainable with $e_2$.

Also, we claim that any o-edge $e$ emanating from $v$ and sharing the first edge with some o-edge in the same orbit as $e_1$ has to be in the same orbit as $e_1$ (Figure~\ref{fig:e-1}). 
\begin{figure}[!htbb]
\centering
\psfrag{a}{$e_1$}
\psfrag{b}{$e_2$}
\psfrag{c}{$e'_1$}
\psfrag{e}{$e$}
\includegraphics[scale=1.2]{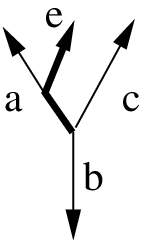}
\caption{}
\label{fig:e-1}
\end{figure}
In fact, if this was not the case we would have o-edges in the orbits of $e_1,e_2,e$ sharing only their common starting point. Hence, this would be a triple as in Case (1) because the orbits of $e_1,e_2,e$ would be pairwise distinct: $e_1$ is not in the same orbit as $e_2$ by hypothesis, and if $e$ was in the same orbit as $e_2$ there would be two o-edges emanating from $v$ in the orbit of $e_2$ that only intersect at $v$, and this would imply that $e_1$ is chainable with $e_2$ (see Remark~\ref{rem:chainable}).

Finally, all o-edges in the orbit of $e_1^{-1}$ emanating from $v$ must share the first edge, for otherwise $e_1$ would be chainable with $e_1^{-1}$ (see Remark~\ref{rem:chainable}).

The three observations above imply that all o-edges in the orbit of $e_1^{-1}$ emanating from $v$ share the first edge with $e_2$, for otherwise $e_1$ would be chainable with $e_2$. 

Also, any other edge emanating from $v$ is contained only in o-edges in the orbit of $e_1$, for otherwise we would have a triple as in Case (1). 
Notice that (once we show that the orientation is well-defined) this gives us the second property above.

Let us now show that any edge $e$ is contained in some o-edge in the orbit of either $e_1$ or $e_1^{-1}$.

By minimality of the action, $e$ is contained in some o-edge $e'$. If $e'$ is in the orbit of either $e_1$ or $e_1^{-1}$ we are done, so suppose not. By what we have shown so far, $e'$ shares the initial edge with an o-edge in the orbit of $e_1^{-1}$. It is readily seen that $e$ is either contained in an o-edge in the orbit of $e_1^{-1}$ or in an o-edge $e''$ (the thickened o-edge in Figure \ref{fig:e'}) that lies in the orbit of $e_1$ since it shares the initial edge with an o-edge in the orbit of $e_1$.

\begin{figure}[!htbb]
\centering
\psfrag{b}{$e_1$}
\psfrag{d}{$e'$}
\psfrag{e}{$e''$}
\includegraphics[scale=0.8]{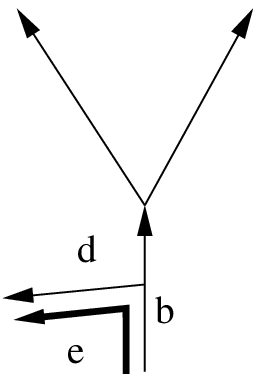}
\caption{}
\label{fig:e'}
\end{figure}

Finally, we are left to show that the orientation is well-defined, namely that any given edge cannot be contained both in an o-edge in the orbit of $e_1$ and an o-edge in the orbit of $e_1^{-1}$. 
If not,  we could construct an o-edge $e'$ (the thickened o-edge in Figure \ref{fig:e_1}) sharing the first edge with $e_1$ and the final edge with an o-edge in the orbit of $e_1^{-1}$. The inverse of such o-edge would be in the orbit of $e_1^{-1}$ but would share the first edge with some o-edge in the orbit of $e_1$, a contradiction.

\begin{figure}[!htbb]
\centering
\psfrag{a}{$e_1$}
\psfrag{b}{$e_1^{-1}$}
\psfrag{c}{$e'$}
\psfrag{e}{$e'$}
\includegraphics[scale=1.2]{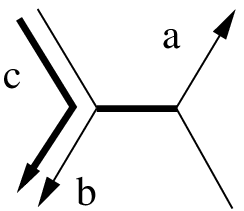}
\caption{}
\label{fig:e_1}
\end{figure}

The orientation we just constructed is clearly $\Gamma$-invariant, and there is only one point at infinity $a$ corresponding to geodesic rays obtained concatenating negatively oriented edges. Such point is a fixed point of $\Gamma$.

\par\medskip
\emph{Case (3a). There exists a minimal integer $k$ so that $\Gamma$ acts transitively
on $\Evone,\dots,\Evkone$ but not on $\Evk$, and also $k>1$.}
\par\medskip

As positive integer we pick $k$.

By hypothesis there exist at least two orbits of 
oriented geodesics of orbit length $k$ in $\TT$.
We pick one orbit of oriented geodesics, and we label  as $a$ 
that orbit, and we label as $b$ some other orbit of oriented geodesics.
Finally, we assign label $c$ to all other orbits.

Let us now realise every word $w$ in $\{a,b,c\}$
as (the label of) an oriented o-geodesic in $\TT$. 
Consider any $o(k)$-geodesic $\xi$
starting from the fixed vertex $v$ with label the first letter of $w$.
(The geodesic segment 
$\xi$ exists because $\Gamma$ is transitive on $\rho(\Gamma) v$.)
Let $\xi_{k-1}$ be the final subpath of $\xi$ of orbit length $k-1$, and
let $\theta$ be an oriented $o(k)$-geodesic containing $\xi_{k-1}$ and labelled as the second letter of $w$ (here we use the transitivity of $\Gamma$ on $o(k-1)$-geodesics). 
We can now concatenate $\xi$ with $\theta_1$, where
$\theta_1$ is the final subpath of $\theta$ of orbit length $1$
(see Figure~\ref{fig:geodesic}).
\begin{figure}[!htbb]
\centering
\psfrag{c}{$\xi$}
\psfrag{d}{$\theta$}
\psfrag{e}{$v$}
\psfrag{f}{$\theta_1$}
\psfrag{g}{$\xi_{k-1}$}
\includegraphics[scale=0.6]{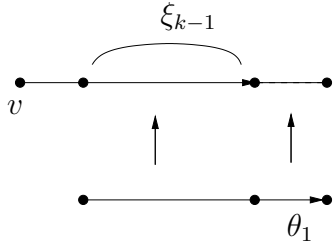}
\caption{Realise a word as an oriented o-geodesic}
\label{fig:geodesic}
\end{figure}
Iterating this construction, we can construct an oriented o-geodesic with any required label.

\par\medskip
\emph{Case (3b). $\Gamma$ acts transitively on $\Evk$ for each $k$.}
\par\medskip

Notice that all o-edges have the same length, since they are in the same orbit. Also, such length cannot be larger than $2$ for otherwise there would exist an o-edge of length $2$, the thickened one in Figure~\ref{fig:o-edges}.
\begin{figure}[!htbb]
\centering
\includegraphics[scale=1.2]{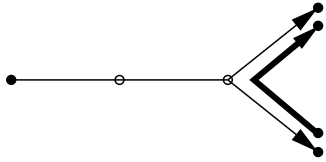}
\caption{o-edges of length $2$ and $3$ in the same orbit}
\label{fig:o-edges}
\end{figure}
Therefore, the transitivity of the $\Gamma$-action on $\Ek$ for every $k>0$ implies that
we are in the situation described in point (i) of Theorem~\ref{thm:classification}.

\subsection{Preparatory lemmas}\label{sub:tec}

From now on we assume that situations (i) and (ii) of Theorem~\ref{thm:classification} do not hold, so that there exists a labelling with the property that all good words can be realised as the label of some o-geodesic. Recall that part of the data of a labelling is an integer $k$, which we fix from now on.

To justify what comes next, we mention now that we want to use the following simple criterion to prove linear independence.

\begin{lemma}\label{homog}
 Let $\{c_n\}$ be non-trivial quasimorphisms on a group $\Gamma$.
Suppose that there exist elements $\{\eta_{n},h_{n}\}$ of $\Gamma$ so that for each positive integers $z,m,n$ we have
 $$c_n(\eta_{m}^z)=c_n(h_{m}^z)=0$$
 and
$$ \begin{cases}
  c_n((\eta_{m}h_{m})^z)=0 & {\rm if\ } m\neq n\\
  c_n((\eta_{m}h_{m})^z)\geq z-1 & {\rm if\ } m= n\\
 \end{cases}$$
 Then $\{[\delta^1c_n]\}$ are linearly independent in $\hb^2(\Gamma,\mathbb{R})$.
\end{lemma}
\begin{proof}
The conclusion easily follows by picking the homogenization of $c_n$.
\end{proof}

\begin{rem}
 We could use the same criterion with ``$=z$'' replacing ``$\geq z-1$'', but that would require being a bit more careful.
\end{rem}

For $w$ a word in $a,b$, we denote $w^{-1}$ the word obtained reading $w$ from right to left and replacing each label of an o(k)-geodesic with the label of the o(k)-geodesic with opposite orientation.

Fix from now on a choice of words $\{w_{ni}\}$ as in Lemma \ref{lem:triple}, as well as $n_0=n_0(k)$.

\begin{lemma}\label{geodoverlap}
 Suppose that for some integers $n,m\geq n_0$, some $i,j\in\{1,2,3\}$ and some $\epsilon_1,\epsilon_2\in\{\pm 1\}$ we have o-geodesics $\gamma_1,\gamma_2$ in $\TT$ labelled, respectively, by $w^{\epsilon_1}_{ni}$ and $w^{\epsilon_2}_{mj}$ that share a common o-subgeodesic of o-length at least $\min\{|\gamma_1|_o,|\gamma_2|_o\}/10-1$. Then $\epsilon_1=\epsilon_2, i=j,m=n$.
\end{lemma}

Notice that $|\gamma_1|_o=|w^{\epsilon_1}_{ni}|+k-1$.

\begin{proof}
Notice that if two o-geodesics share a common subgeodesic of o-length, say $L$, then their labels have a common subword of length $L-k+1$.

 If $\epsilon_1=\epsilon_2$, then it is easy to get $i=j,m=n$ because $w_{ni}$ and $w_{mj}$ can share a long subword only if they coincide.
 
 We will now argue that we cannot have, say $\epsilon_1=+1$ and $\epsilon_2=-1$. Suppose by contradiction that this is the case. Notice that by Lemma \ref{lem:triple}-(iii) there must be an o(k)-geodesic labelled $a$ or $b$ (contained in $\gamma_1$) that coincides with the inverse of an o(k)-geodesic labelled, respectively, $b$ or $a$ (contained in $\gamma_2$). In the second case of Remark \ref{rem:inverses}, this cannot happen, hence we can suppose that the labels $a,b$ correspond to one orbit of o(k)-geodesics each. In this case we get that the label $w'_{mj}$ of $\gamma_2$ is obtained reading $w_{mj}$ right-to-left and replacing each $a$ with $b$ and vice versa. It is easily seen that $w'_{mj}$ and $w_{nj}$ cannot share a long subword in view of Lemma \ref{lem:triple}-(i).
\end{proof}


\begin{defi}
Let us call \emph{$w^{\pm 1}_{ni}$-subgeodesic} an o-subgeodesic of an o-geodesic labelled 
$w^{\pm 1}_{ni}$. A  \emph{long $w^{\pm 1}_{ni}$-subgeodesic} is  a $w^{\pm 1}_{ni}$-subgeodesic of o-length at least $(|w_{ni}|+k-1)/2$.
\end{defi}
\begin{defi}
The \emph{almost concatenation} of the o-geodesics $\gamma_1$ and $\gamma_2$ (if it exists)
is the o-geodesic obtained concatenating either $\gamma_1$ and $\gamma_2$   
or $\gamma_1$, an o-edge and $\gamma_2$. Almost concatenation of more than two o-geodesics can be defined similarly.
\end{defi}

\begin{lemma}\label{actuallabel}
 There exist elements $g_{ni}\in \Gamma$ so that the following hold. 
 \begin{enumerate}
\item There exists an o-geodesic $s_n$ obtained almost concatenating a long $w_{n3}$-subgeodesic and a long $w_{n1}$-subgeodesic with the following property. 
For each positive integer $N$, the o-geodesic  from $v$ to $(g_{n1}g_{n3})^Nv$ is the almost concatenation of a long $w_{n1}$-subgeodesic, $N-1$ translates of $s_n$ and a long $w_{n3}$-subgeodesic.
  \item For each positive integer $N$, the o-geodesic from $v$ to $(g_{n1}g_{n2})^Nv$ (resp. $(g_{n2}^{-1}g_{n3})^Nv$) is obtained alternately almost concatenating  long $w_{n1}$-subgeodesics and  long $w_{n2}$-subgeodesics (resp.~long $w_{n2}^{-1}$-subgeodesic  and  long $w_{n3}$-subgeodesics).
 \end{enumerate}
\end{lemma}

\begin{proof}

\begin{figure}[!htbb]
\centering
\psfrag{w3}{$w_{n3}$}
\psfrag{w2}{$w_{n1}$}
\includegraphics[scale=0.8]{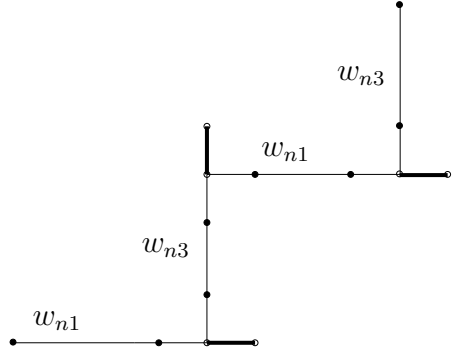}
\caption{the o-geodesic $[v,(g_{n1}g_{n3})^2v]$}
\label{concatenation}
\end{figure}
 Chose $g_{ni}$ so that the geodesic from $v$ to $g_{ni}v$ is labelled $w_{ni}$. Let us study the o-geodesic $[v,(g_{n1}g_{n3})^2v]$, the other cases being similar. We can form a path from $v$ to $(g_{n1}g_{n3})^2v$ by concatenating $4$ geodesics, alternately labelled $w_{n1}$ and $w_{n3}$. Since by Lemma \ref{geodoverlap} we can bound the overlap of geodesics labelled $w_{n1}$ (resp. $w_{n3}$) and geodesics labelled $w_{n3}^{-1}$ (resp. $w_{n1}^{-1}$), the conclusion easily follows, see Figure \ref{concatenation}.
\end{proof}

\subsection{Conclusion of proof of Theorem~\ref{thm:classification}}

Let $s_n$ be as in Lemma~\ref{actuallabel}. We want to apply Lemma \ref{homog} to $c_n=f_{\rho(\Gamma) s_n}$, $\eta_n=g_{n1}g_{n2}$, $h_n=g_{n2}^{-1}g_{n3}$ for $n$ sufficiently large. Let us for example show $c_n((g_{m2}^{-1}g_{m3})^z)=0$, the other cases being similar (it is also worthwhile to point out that the geodesic $[v,(g_{n1}g_{n3})^zv]$ contains at least $z-1$ translates of $s_n$ by Lemma \ref{actuallabel}-(1)). If there was a translate of $s_n^{\pm1}$ contained in $[v,(g_{m2}^{-1}g_{m3})^zv]$, then a suitable translate of $s_n$ would contain a $w_{n1}^{\pm 1}$-subgeodesic $\gamma$ of length $(|w^{\pm1}_{n1}|+k-1)/4-1$ with one of the following properties:
\begin{itemize}
\item $\gamma$ is contained either in a long $w_{m2}^{-1}$-subgeodesic or in a long $w_{m3}$-subgeodesic, or
\item $\gamma$ contains either a long $w_{m2}^{-1}$-subgeodesic or a long $w_{m3}$-subgeodesic.
\end{itemize}
All configurations are excluded by Lemma~\ref{geodoverlap}.
By definition of $f_{\rho(\Gamma) s_n}$, this implies $c_n((g_{m2}^{-1}g_{m3})^z)=0$.\qed

\section{Remarks on actions fixing a point at infinity}\label{remarks}

In this section we make an observation about actions as in Theorem \ref{thm:classification}-(ii).

For $a\in \partial \TT$ we denote by $\beta_a$ any Busemann function based at $a$ so that, for any vertex $w$ of $\TT$, $\beta_a(w)$ is an integer. Also, for $a\in\partial \TT$ we denote by $\En_\beta$ the set of \emph{monotone geodesics} of length $n$, meaning geodesics of length $n$ along which $\beta_a$ is increasing.

\begin{prop}\label{prop:busemann}
 Suppose that $\Gamma$ acts minimally on the tree $\TT$ and that every vertex of $\TT$ has valence greater than 2. Suppose that $\Gamma$ fixes a point $a$ in the boundary $\partial \TT$ of $\TT$.
Then there exists an integer $l$ so that, for each integer $n$, $\Gamma$ acts transitively on
$$\{[x,y]\in\En_\beta: \beta_a(x) \equiv 0 \, (l)\}.$$
\end{prop}

\begin{proof}
 First of all, let us show that there exists some $h\in\Gamma$ that acts hyperbolically on $\TT$. Pick any $v\in\TT$. It is easily seen that any element $h\in\Gamma$ so that $\beta_a(hv)\neq \beta_a(v)$ acts hyperbolically, because, up to exchanging $v$ and $hv$ and replacing $h$ with $h^{-1}$, a subray of $[v,a)$ gets mapped by $h$ into a proper subray. Since there are no leaves in $\TT$, such $h$ must exist.
 
 In order to prove the proposition, we can fix a basepoint $v$ on the axis of $h$ and show that for any integer $k$, $\Gamma$ acts transitively on the set of monotone $o(k)$-geodesics. This is clear when  we fix the axis of a conjugate of $h$ and we restrict to monotone $o(k)$-geodesics contained in that axis. Notice now that for any $g\in\Gamma$, the axis of $h^g$ passes through $gv$ and shares a common subray with the axis of $h$. Hence, for any $g\in\Gamma$, the monotone $o(k)$-geodesic starting at $gv$ is in the same orbit as a monotone $o(k)$-geodesic contained in the axis of $h$, and hence it is in the same orbit as the monotone $o(k)$-geodesic starting at $v$.
\end{proof}

\section{Local $\infty$-transitivity and quasimorphisms}\label{sec:transitivity}
In this section we prove Corollary~\ref{cor:locally_inf_trans}.
We start with the following easy observation:

\begin{lemma}\label{lem:transitive_on_spheres}  Let $\TT$ be a simplicial tree and let $\Gamma<\aut(\TT)$.
Let us suppose that we are in the situation (i) of Theorem~\ref{thm:classification}, 
that is there exists $l\in \{1,2\}$ so that for each $n$, $\Gamma$ acts transitively on
$$\{[x,y]\in\En: \ d(x,v) \equiv 0\, (l)\},$$
where $v$ is any vertex of $\TT$.
Then there exists a vertex $x_0$ such that 
the stabiliser $\stab_\Gamma(x_0)$ acts transitively on all spheres centered at $x_0$.
\end{lemma}

\begin{proof} 
Taking $x_0:=v$ we have that, for any $n$, $\Gamma$ is transitive on paths of length $n$ starting at $x_0$.
Hence the stabiliser $\stab_\Gamma(x_0)$ acts transitively on the sphere of radius $n$
centered at $x_0$.
\end{proof}

Let $\TT$ be a locally finite tree with $\VV$ the set of vertices and $\EE$ the set of edges.
If $e\in\EE$, we denote by $\TT_e^+$ 
the connected component of $\TT\smallsetminus\{e\}$
containing the terminus $t(e)$ of $e$. 

\begin{lemma}\label{lem:cofinite+nofixedpoints=minimal}  
Let $\TT$ be a locally finite tree such that each vertex has valence at least $2$.  Let $\Lambda<\aut(\TT)$ be
a subgroup such that $\Lambda\backslash\TT$ is finite and there is no fixed point in the boundary $\partial\TT$.
Then $\Lambda$ acts minimally on $\TT$.
\end{lemma}

\begin{proof}  We will show first that the $\Lambda$-action on the boundary $\partial\TT$
is minimal, and conclude that the action on the tree is minimal as well.

Let $F\subseteq\partial\TT$ be a closed non-empty $\Lambda$-invariant subset.
We suppose that $F\neq\partial\TT$.
Let $\TT'$ be the convex hull of $F$, that is the subtree consisting of all bi-infinite geodesics connecting two
elements in $F$.  Let $\xi\in\partial\TT\smallsetminus F$ and let $e\in \EE$ be an edge in $\TT$ such that 
$\xi\in\partial\TT_e^+\subset\partial\TT\smallsetminus F$.  
Let $v_n\in \TT^+_e$ be a vertex with $d(v_n,t(e))=n$.  Since $\TT_e^+$ and $\TT'$ are disjoint,
$d(\TT',v_n)\geq n$ and, by invariance, $d(\TT',\lambda v_n)\geq n$ for all $\lambda\in\Lambda$.
Let $D\subset\VV$ be a fundamental domain for the $\Lambda$-action in $\VV$.
Since $D$ is finite, there exists $n_0\in\mathbb N$ such that $d(v,\TT')\leq n_0$ for all $v\in D$.
By choosing $n$ with $n>n_0$ we see that $D$ cannot intersect the $\Lambda$-orbit of $v_n$,
which is a contradiction.  Thus $\Lambda$ acts minimally on $\partial\TT$.

Let now $\TT'\subset\TT$ be a $\Lambda$-invariant non-empty subtree.
Then $\partial\TT'\subseteq\partial\TT$ is a closed non-empty $\Lambda$-invariant subset
and we will show that $\partial\TT'\neq\partial\TT$.  
Since $\TT'\neq\TT$, there exists a vertex $v\in\TT'$ 
such that the valence of $v$ in $\TT'$ is strictly smaller than the valence of $v$ in $\TT$.
Let $v'\in\VV(\TT)$ be an adjacent vertex with $v'\notin\VV(\TT')$
and let $e:=(v,v')$ be the corresponding oriented edge in $\TT$.
Since the valence of each vertex is at least $2$, this edge can be extended
to an infinite ray, so that $\partial\TT_e^+\neq\varnothing$.  
But $\TT_e^+\cap\TT'=\varnothing$, so that $\partial\TT_e^+\cap\partial\TT'=\varnothing$
and hence $\partial\TT'\neq\partial\TT$.
\end{proof}

In order to prove Corollary~\ref{cor:locally_inf_trans} , we recall the following lemma from \cite{Burger_Mozes_gat}:

\begin{lemma}[{\cite[Lemma~3.1.1]{Burger_Mozes_gat}}]\label{lem:3.3.1}  Let $\TT$ be a regular tree with vertex set $\VV$
and $H<\aut(\TT)$ a closed subgroup.  Then the following assertions are equivalent:
\begin{enumerate}
\item $H$ is locally $\infty$-transitive;
\item $\stab_H(v)$ is transitive on $\partial\TT$ for every vertex $v\in \VV$;
\item $H$ is non-compact and acts transitively on $\partial\TT$;
\item $H$ is doubly transitive on $\partial\TT$. 
\end{enumerate}
\end{lemma}

\begin{proof}[Proof of Corollary~\ref{cor:locally_inf_trans}]  The implication (1)$\Rightarrow$(2)
is in \cite[Corollary~26]{Burger_Monod_GAFA}, while (2)$\Rightarrow$(3) is obvious.

To show that (3)$\Rightarrow$(1) observe that since $\TT_1$ and $\TT_2$ are regular 
and $\Gamma<\aut(\TT_1)\times\aut(\TT_2)$ is cocompact, then $\Gamma$-action on $\TT_i$ is cofinite.
If $H_i:=\overline{\pr_i(\Gamma)}^Z$, the projection $\pr_i:\aut(\TT_1)\times\aut(\TT_2)\to\aut(\TT_i)$
induces an equivariant map $(\aut(\TT_1)\times\aut(\TT_2))/\Gamma\to\aut(\TT_i)/H_i$,
so that on $\aut(\TT_i)/H_i$ there is a finite invariant measure.  
If $H_i$ were to fix a point $\xi\in\partial\TT_i$, it would be amenable 
(see for example \cite[Chapter~1, Section~8]{FigaTalamanca_Nebbia}) and because of the existence of a finite
invariant measure on $\aut(\TT_i)/H_i$, $\aut(\TT_i)$ would be amenable as well.
Hence $H_i$ does not fix any point $\xi\in\partial\TT_i$,
so that by Lemma~\ref{lem:cofinite+nofixedpoints=minimal} the action on $\TT_i$ is minimal.
Thus the $\Gamma$-action on $\TT_i$ satisfies the hypothesis of Theorem~\ref{thm:classification}
and we must be in case (i).  By Lemma~\ref{lem:transitive_on_spheres} there exists a vertex $x\in\TT_i$ such that 
$\stab_\Gamma(x)$ acts transitively on all spheres centered at $x$.  
By continuity $\stab_{H_i}(x)$ acts transitively on $\partial\TT_i$ and hence so does $H_i$.
Since $H_i$ is non-compact, by the implication (3)$\Rightarrow$(1) of Lemma~\ref{lem:3.3.1},
$H_i$ is locally $\infty$-transitive.
\end{proof}

\section{Examples}\label{examples}
We point out here that our Theorem~\ref{thm:classification} is not covered by the results
of Fujiwara~\cite{Fujiwara} and of Minasyan and Osin~\cite{MO}.

\subsection{Amalgamated free product of groups}
Theorem 1.1~\cite{Fujiwara} states that if  $G=A\ast_C B$ is such that $|C\backslash A/C|\geq 3$
and $|B/C|\geq 2$ then there exists an injective $\mathbb{R}$-linear map $\ell^1\rightarrow \hb^2(G,\mathbb{R})$.

We now show how to deduce Fujiwara's result from Theorem \ref{thm:classification}.

According to Bass-Serre theory, $G$ acts on a tree $\TT$ whose vertices are the left $G$-cosets of $A$ and $B$ and whose edges are the left $G$-cosets of $C$.

If $|B/C|=2$, as explained in Remark~\ref{val2}, we replace $\TT$ with the tree $\TT'$ obtained merging edges that share a vertex of degree 2, if not we set $\TT=\TT'$. Then $\TT'$ satisfies the conditions of Theorem~\ref{thm:classification}. We now claim that if $1,a_1,a_2\in A$ lie in pairwise distinct double-cosets, then the oriented geodesics $[B,a_1B],[B,a_2B]$ in $\TT$ are not in the same $G$-orbit. This easily implies that $\TT'$ has at least two orbits of geodesics of length 2 and hence Theorem~\ref{thm:classification}-(i) does not hold for $\TT'$. 
To prove the claim, just observe that an element $g\in G$ mapping $[B,a_1B]$ to $[B,a_2B]$ stabilises $[B,A]$, and hence belongs to $C$, and maps the edge labelled $a_1C$ to the one labelled $a_2C$, whence it also belongs to $a_2Ca_1^{-1}$, which contradicts the fact that $a_1,a_2$ are not in the same double coset (see Figure~\ref{tree}). 
\begin{figure}[!htbb]
\centering
\psfrag{a}{$A$}
\psfrag{b}{$B$}
\psfrag{b1}{$a_1B$}
\psfrag{b2}{$a_2B$}
\psfrag{c}{$C$}
\psfrag{c1}{$a_1C$}
\psfrag{c2}{$a_2C$}
\includegraphics[scale=1.2]{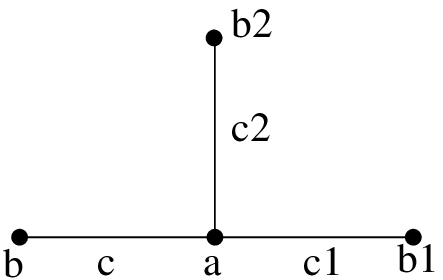}
\caption{}
\label{tree}
\end{figure}
Theorem~\ref{thm:classification}-(ii) is well-known not to hold (for $\TT'$), 
for example because the stabiliser in $\TT$ of each $A$-vertex acts transitively on the edges emanating from it. 
Therefore, condition (iii) of Theorem~\ref{thm:classification} holds and 
we can deduce that  $\hb^2(G,\mathbb{R})$ contains a copy of $\ell^1$.

Our result is more general even when restricted to amalgamated products. Indeed, let us consider
the group $S_3\ast_{\mathbb{Z}_2}\mathbb{Z}_4$ 
where $S_3$ is the symmetric group on $3$ elements and its associated tree $\TT$. 
Since the stabilizer of each vertex of $\TT$ is a finite group, it easily follows that
conditions (i) and (ii) of Theorem~\ref{thm:classification} cannot be satisfied. 
Hence, $\hb^2(S_3\ast_{\mathbb{Z}_2}\mathbb{Z}_4,\mathbb{R})$ is infinite dimensional, 
a fact that can also be deduced from the fact that it  virtually free. 
However, it cannot be deduced from Fujiwara's result because 
$|\mathbb{Z}_2\backslash S_3/\mathbb{Z}_2|\geq 3$ does not hold.

\subsubsection{A family of examples}  Based on \cite[Example 1.2.1]{Burger_Mozes_gat} we can construct an infinite family of examples
of (cocompact) irreducible lattices $\Gamma<\PSL(n,\QQ_{p_1})\times\PSL(n,\QQ_{p_2})$,
where $p_1,p_2$ are primes and $n\geq3$.  We illustrate the example for $n=3$.

For $\ell\in\{p_1,p_2\}$ we consider the Bruhat--Tits building $\Delta_\ell$ associated to $\PSL(3,\QQ_\ell)$,
that we proceed to recall (see \cite[Chapter~V, 8B]{Brown}).  Let $V$ be a $3$-dimensional vector space over $\QQ_\ell$.
Let $\Ll'$ be the space of lattices $L\subset V$, 
that is of sub-$\ZZ_\ell$-modules $L=\ZZ_\ell v_1+\ZZ_\ell v_2+\ZZ_\ell v_3$,
where $\{v_1,v_2,v_3\}$ is a $\QQ_\ell$-basis of $V$. 
We consider on $\Ll'$ an equivalence relation where $L_1\sim L_2$ if
there exists $\lambda\in\QQ_\ell^*$ such that $L_1=\lambda L_2$,
and let $\Ll:=\Ll'/\sim$.  On $\Ll$ we define an incidence relation, where we
say that $[L_1], [L_2]\in\Ll$ are \emph{incident} if $[L_1]\neq[L_2]$ and 
there exists representatives $L_1,L_2$
such that $\ell L_1\subset L_2\subset L_1$.  Consider now the \emph{flag complex},
that is the simplicial
complex whose $0$-cells are equivalence classes of lattices $[L]\in\Ll$,
whose $1$-cells are pairs of incident $0$-cells and 
whose $2$-cells are triples of pairwise incident $0$-cells.
We denote  $\Delta_\ell$ the flag complex with the natural $\PSL(3,\QQ_\ell)$-action.

To an $0$-cell $[L]\in\Ll$ we can associate a \emph{type}
as follows.  Let $L$ be a representative of $[L]$ with 
$L=\ZZ_\ell v_1+\ZZ_\ell v_2+\ZZ_\ell v_3$; if $\det(v_1,v_2,v_3)=\ell^nk$,
where $k\in\ZZ_\ell^\times$, $n\in\ZZ$, then the type of $[L]$ is $n\mod3$.
It is easy to see that the type of an $0$-cell is well defined.  As an example,
the type of 
$\ZZ_\ell e_1+\ZZ_\ell e_2+\ZZ_\ell e_3$
is $0$, the type of
$\ZZ_\ell e_1+\ZZ_\ell e_2+\ZZ_\ell\ell e_3$
is $1$, while the type of 
$\ZZ_\ell e_1+\ZZ_\ell\ell e_2+\ZZ_\ell\ell e_3$
is $2$,
where $e_1,e_2,e_3$ are the standard basis elements.
In addition to $3$ types of $0$-cells, there are also $3$-types of $1$-cells.
Let $\Gg_\ell$ be the subgraph of $\Delta_\ell$ consisting of edges of a given type, for example $(02)$.
Fix a vertex $[L_0]$ in $\Gg_\ell$, for example of type $0$, $L_0=\ZZ_\ell e_1+\ZZ_\ell e_2+\ZZ_\ell e_3$.
The vertices $[L]$ incident to $[L_0]$ are all of type $2$ and, by definition, all satisfy the condition
that $\ell L_0\subset L\subset L_0$.  In other words they are in one-to-one correspondence with 
subspaces (in this case lines) in $L_0/\ell L_0\simeq\FF_\ell^3$.
The stabiliser of $[L_0]$ in $\PSL(3,\QQ_\ell)$ is $\PSL(3,\ZZ_\ell)$ and the effective action 
of $\stab_{\PSL(3,\QQ_\ell)}([L_0])=\PSL(3,\ZZ_\ell)$ on the $\ell^2+\ell+1$ vertices incident to $[L_0]$ 
is the action of $\PSL(3,\FF_\ell)$ on the $\ell^2+\ell+1$ lines in $\FF_\ell^3$.
Since $\PSL(3,\FF_\ell)$ is doubly transitive on the set of pairs of lines, 
it follows that $\stab_{\PSL(3,\QQ_\ell)}([L_0])$ is doubly transitive on the set of 
incident vertices. It is easily shown that, given an action on a tree, if all vertex stabilisers are doubly transitive on the sphere of radius 1 around the corresponding point, then all stabilisers act transitively on the sphere of radius 2.

We have hence given a sketch of the proof the following:

\begin{lemma}\label{lem:two-transitive}  The $\PSL(3,\QQ_\ell)$-action on the graph $\Gg_\ell$
is locally $2$-transitive.
\end{lemma}

Let $\TT_{\ell^2+\ell+1}$ be the regular tree that is the universal covering of $\Gg_\ell$.
Let 
\begin{equation*}
\xymatrix{
0\ar[r]
&\pi_1(\Gg_\ell)\ar@{^(->}[r]
&H_\ell\ar@{->>}[r]
&\PSL(3,\QQ_\ell)\ar[r]
&0
}
\end{equation*}
be the exact sequence associated to the universal covering projection
$\pi:\TT_{\ell^2+\ell+1}\to\Gg_\ell$, that is
\begin{equation*}
H_\ell:=\{g\in\aut(\TT_{\ell^2+\ell+1}):\,\pi\circ g= h\circ\pi\text{ for some }h\in\PSL(3,\QQ_\ell)\}\,,
\end{equation*}
with
\begin{equation}\label{eq:normal}
\pi_1(\Gg_\ell)=\{g\in\aut(\TT_{\ell^2+\ell+1}):\,\pi\circ g= \pi\}\triangleleft H_\ell\,.
\end{equation}

The following lemma shows that $H_\ell$ cannot be locally $\infty$-transitive.

\begin{lemma}\label{lem:locally_infty_transitive}  Let $\TT$ be a locally finite tree and let $H<\aut(\TT)$ 
be a closed locally $\infty$-transitive subgroup.  Any normal discrete subgroup $N\triangleleft H$
consisting of hyperbolic elements must be trivial.
\end{lemma}

\begin{proof} Since it is discrete, the group $N$ is countable and, in particular, the set of axes $\{a_n:\,n\in N,n\neq e\}$
is countable.  If $h\in H$, we have that $h(a_n)=a_{hnh^{-1}}$.  But Lemma~\ref{lem:3.3.1} (1)$\Rightarrow$(4) implies that
$H$ acts doubly transitive on $\partial\TT$, so that $\{h(a_n):\,h\in H\}$ would be uncountable.
\end{proof}

Let us consider now an irreducible (cocompact) lattice 
$$\Gamma<\PSL(3,\QQ_{p_1})\times\PSL(3,\QQ_{p_2})$$
and let us consider the inverse image $\widetilde\Gamma<H_{p_1}\times H_{p_2}$ of $\Gamma$ with respect to the exact sequence
\begin{equation*}
\xymatrix@1{
\pi_1(\Gg_{p_1})\times\pi_1(\Gg_{p_2})\,\ar@{^(->}[r]
&H_{p_1}\times H_{p_2}\ar@{->>}[r]
&\PSL(3,\QQ_{p_1})\times\PSL(3,\QQ_{p_2})
}
\end{equation*}
Since $\Gamma$ is irreducible, it has dense projections and hence, by Lemma~\ref{lem:two-transitive}
the action of $\pr_i(\Gamma)$ on $\Gg_{p_i}$ is locally $2$-transitive, for $i=1,2$.  
It follows that also the action of $\pr_i(\widetilde\Gamma)$ on $\TT_{p_i^2+p_i+1}$ is locally $2$-transitive;
however by Lemma~\ref{lem:locally_infty_transitive}, the closures of these projections are not locally $\infty$-transitive.

By Corollary~\ref{cor:locally_inf_trans} we deduce that $\widetilde\Gamma$
has an infinite dimensional set of linearly independent median quasimorphisms.  Notice that $\Gamma$,
being an irreducible lattice in a high rank Lie group, has no non-trivial quasimorphisms, \cite{Burger_Monod_GAFA}.

The same construction can of course being performed for an irreducible lattice 
$\Gamma<\PSL(n,\QQ_{p_1})\times\PSL(n,\QQ_{p_2})$, for any $n\geq3$, in which case
we will have an $(n-1)$-dimensional flag complex.

Notice that Fujiwara's condition on the number of double cosets is in particular incompatible
with the local $2$-transitivity of $\widetilde\Gamma$.  
In fact, if we realise 
$\widetilde\Gamma$ as the product of $\stab_{\widetilde\Gamma}(v_1)$ and $\stab_{\widetilde\Gamma}(v_2)$,
amalgamated over $\stab_{\widetilde\Gamma}(v_1,v_2)$,  the local $2$-transitivity is equivalent to 
$\big|\stab_{\widetilde\Gamma}(v_1,v_2)\backslash\stab_{\widetilde\Gamma}(v_i)/\stab_{\widetilde\Gamma}(v_1,v_2)\big|=2$,
for $i=1,2$.

\subsection{Acylindrical hyperbolicity}
Minasyan and Osin~\cite{MO} gave a criterion for groups acting on a tree to be acylindrically hyperbolic (see~\cite{Osin} for the definition and equivalent characterizations). More precisely, they state that if a group $G$ acts minimally on a simplicial tree $\TT$, it  does not fix any point of $\partial \TT$, and there exist vertices
$u$, $v$ of $\TT$ such that the pointwise stabilizer $\text{PStab}_G\{u, v\}$
is finite, then $G$ is either virtually cyclic or acylindrically hyperbolic.
In particular, if the group is acylindrically hyperbolic then $\hb^2(G,\mathbb{R})$ is infinite dimensional~\cite{HullOsin}.

Condition (iii) of Theorem~\ref{thm:classification} includes groups which are not covered by \cite{MO}. For instance, let us consider the Caley graph $\TT$ of $\mathbb{F}_3=\langle a, b, c \rangle$. Let $G_{bc}$ be an infinite subgroup of $\aut(\TT)$ acting on the subtree $\TT_{bc}$ containing the vertex corresponding to the identity and edges labelled by $b$ and $c$ only, and suppose that $G_{bc}$ fixes $\TT\setminus \TT_{bc}$.
Let $\Gamma$ be the group of automorphisms generated by $\mathbb{F}_3$ and $G_{bc}$.
By construction there exist at least two orbits of edges and $\Gamma$ does not fix a point in $\partial \TT$, therefore condition (iii) of Theorem~\ref{thm:classification} holds.
On the other hand it is easily seen that the stabilizer of each geodesic contains a conjugate of $G_{bc}$ and is therefore infinite, so that Theorem 2.1 in~\cite{MO} does not apply.

\bibliographystyle{alpha}
\bibliography{biblio}

\end{document}